\documentclass[12pt]{amsart}

\usepackage{amssymb,amsfonts}
\usepackage{amsmath,amsthm,amscd}
\usepackage{mathrsfs} 
\usepackage{enumerate}
\usepackage{verbatim}
\usepackage{color}
\usepackage[colorlinks]{hyperref}
\usepackage[displaymath, tightpage]{preview}

\usepackage{beton}

\usepackage[T1]{fontenc}
\usepackage[utf8x]{inputenx} 

\newtheorem{thm}{Theorem}[section]

\newtheorem{lem}[thm]{Lemma}

\newtheorem{prop}[thm]{Proposition}
\newtheorem{fact}[thm]{Fact}
\newtheorem{prob}[thm]{Problem}

\theoremstyle{remark}

\newtheorem{example}[thm]{Example}

\begin{document}

\title{Measures and fibers}

\subjclass[2010]{03E35,03E75,28A60}
\keywords{metrizably-fibered spaces, Suslinean spaces, Martin's Axiom, non-separable measures, countably determined measures, Radon measures}
\thanks{The author was partially supported by National Science Center grant no. 2013/11/B/ST1/03596 (2014-2017).}

\author[Piotr Borodulin--Nadzieja]{Piotr Borodulin-Nadzieja}
\address{Instytut Matematyczny, Uniwersytet Wroc\l awski}
\email{pborod@math.uni.wroc.pl}

\date{24 February 2016}

\begin{abstract} 
We study measures on compact spaces by analyzing the properties of fibers of continuous mappings into $2^\omega$. We show that if a compact zerodimensional space $K$ carries a measure of uncountable Maharam type, then such a mapping has a
non-scattered fiber and, if we assume additionally a weak version of Martin's Axiom, such a mapping has a fiber carrying a measure of uncountable Maharam type.	Also, we prove that every compact zerodimensional space which supports a strictly positive
measure and which can be mapped into $2^\omega$ by a finite-to-one function is separable.
\end{abstract}

\maketitle

\section{Introduction}

We say that a compact space $K$ is \emph{$\varphi$-fibered} (or \emph{has fibers satisfying $\varphi$}) if there is a compact metric space $M$ and a continuous function $f\colon K\to M$ such that $f^{-1}[\{x\}]$ has property $\varphi$ for each $x\in M$. So, we can consider e.g. $n$-fibered,
finitely-fibered, metrizably-fibered or spaces with scattered fibers.

Questions about properties of fibers of continuous mappings to metric spaces appear quite naturally in many contexts. The question if it is consistent that all
perfectly normal compact spaces are $2$-fibered is one of the most important questions of set theoretic topology (see \cite{Gruenhage-problems}). Metrizably- and finitely-fibered spaces were considered in the context of spaces with a small diagonal
(see e.g. \cite{Gruenhage}, \cite{Dow-Hart-diagonal}) and of Rosenthal compacta (see e.g. \cite{Kubis}). Non-separable linearly-fibered spaces are in a sense direct generalizations of the Suslin line and were studied by Moore (\cite{Moore}) and Todor\v{c}evi\'{c} (\cite{Todorcevic}, see also Example \ref{Stevo}).

The systematic study of metrizably-fibered spaces was undertaken in \cite{Tkachuk}. Independently, Tkachenko (\cite{Tkachenko}) considered a slightly weaker notion than being $\varphi$-fibered. We say that
a space $K$ is \emph{$\varphi$-approximable} if there is a
countable cover by closed sets such that the maximal intersections of elements of this cover have property $\varphi$ (by replacing ``closed sets'' with ``zero sets'' we obtain an equivalent definition of $\varphi$-fibered space, see \cite[Proposition 2.1]{Tkachuk}).
Tkachenko investigated e.g. for which properties $\varphi$ spaces satisfying $\varphi$ are $\varphi$-approximable.

In this article we will examine how the existence of certain types of measures affect the properties of fibers. For the sake of this section say that
a compact space $K$ has property $A(\kappa)$ if it can be mapped continuously onto $[0,1]^\kappa$ and $K$ has property $M(\kappa)$ if it carries a measure of Maharam type $\kappa$ (see Section \ref{non-separable} for the definitions). Loosely
speaking by determining for which $\kappa$ a compact space has $A(\kappa)$ we can measure its \emph{combinatorial complexity} and similarly using $M(\kappa)$ we can measure its \emph{measure-theoretic complexity}. 

Note that for each $\kappa$ the property $A(\kappa)$ implies $M(\kappa)$. In general, the converse
holds only for some $\kappa$'s. For example, the scattered spaces (i.e. those which do not have $A(\omega)$) carry only purely atomic measures and so they do not have $M(\omega)$. Fremlin proved that under $\mathsf{MA}_{\omega_1}$ the property $A(\omega_1)$ is equivalent to
$M(\omega_1)$ and there are many consistent examples of spaces with $M(
\omega_1)$ and without $A(\omega_1)$.
For the results concerning the relationship between $A(\kappa)$ and $M(\kappa)$ we refer the reader e.g. to \cite{Plebanek02}. In this article we will be interested only in cases when $\kappa\leq \omega_1$.

Let $K$ be a compact zerodimensional space and let $f\colon K\to 2^\omega$ be a continuous mapping. Of course, $A(\omega)$ is not necessarily inherited by fibers of $f$, since $K$ having $A(\omega)$ can be even $1$-fibered. However, Tkachenko proved
that if $K$ has $A(\omega_1)$, then one of the fibers of $f$ has
$A(\omega_1)$ (see \cite{Tkachenko}, we reprove this result in Section \ref{non-separable}). 

The property $M(\omega_1)$ in this context behaves in a more complicated way. 
If $\mathrm{cov}(\mathcal{N})=\omega_1$, then there is a compact space with $M(\omega_1)$ which has fibers homeomorphic to $2^\omega$, and so none of them has $M(\omega_1)$ (see Example \ref{kunen}). On the other hand, the theorems of Fremlin and of
Tkachenko mentioned above
imply that under $\mathsf{MA}_{\omega_1}$ the property $M(\omega_1)$ has to be inherited by some fiber. 
In Section
\ref{non-separable} we prove that under $\mathsf{MA}_{\omega_1}$ for measure algebras ($\mathsf{MA}_{\omega_1}(\mathrm{ma})$, in short) every zerodimensional space with $M(\omega_1)$ has a fiber with $M(\omega_1)$. Note that $\mathsf{MA}_{\omega_1}(\mathrm{ma})$ is considerably weaker than
$\mathsf{MA}_{\omega_1}$. In particular it is consistent with the existence of space with $M(\omega_1)$ and without $A(\omega_1)$. On the other hand, $\mathsf{MA}_{\omega_1}(\mathrm{ma})$ is equivalent to
$\mathrm{cov}(\mathcal{N}_{\omega_1})>\omega_1$, which is just a slightly stronger axiom than $\mathrm{cov}(\mathcal{N})>\omega_1$.

In Section \ref{scattered} we show in $\mathsf{ZFC}$ that the zerodimensional spaces with $M(\omega_1)$ cannot have too \emph{simple} fibers. Namely, for every continuous $f\colon K\to 2^\omega$, where $K$ has $M(\omega_1)$, at least one fiber of
$f$ has $M(\omega)$. There are many properties of compact spaces implying that a space carries only
separable measures (e.g. linearity, being a Rosenthal compactum, being a Stone space of a minimally generated Boolean algebra). Theorem \ref{perfect fiber} adds the property of being a space with scattered fibers to this list. 
In Section \ref{examples} we provide some known examples of spaces with scattered fibers. Theorem \ref{perfect fiber} implies that all of them carry only separable measures.

Section \ref{scattered} contains one more results of this sort. Say that a compact space has property ($\star$) if it is non-separable and supports a (strictly positive) measure. In a sense such spaces are big from the measure-theoretic point of view, similarly to spaces having $M(\omega_1)$. There
are spaces having $M(\omega_1)$ but not ($\star$) (e.g. $2^{\omega_1}$). The question if ($\star$) implies $M(\omega_1)$ is quite interesting and, in fact, it was one of the motivations for our research (see \cite{Pbn-Grzes-16} for further discussion). It turns out that there are, at least consistently,
non-separable zerodimensional spaces supporting measures and having scattered fibers (see remarks at the end of Section \ref{examples}). In light of Theorem \ref{perfect fiber} those spaces have ($\star$) but not $M(\omega_1)$. In Section \ref{scattered}
we prove Theorem \ref{fin fibered} saying that the spaces with ($\star$) at least cannot be finitely-fibered. 

In Section \ref{examples} we present some relevant examples and we give some additional motivation for our research.


Let us mention that although the authors mentioned at the beginning of this section have not considered properties connected to measures in the context of fibers, the study of measures on fibers of measurable mappings is an important part of probability theory. It is
worth
to recall here well-known Rokhlin's theorem:
\begin{thm}[\cite{Rokhlin}]
	Let $K$ be a separable compact space and let $f\colon K\to M$ be a measurable mapping into the compact metric space $M$. If $\mu$ is a measure on $K$, then there is a family of measures $\{\mu_x\colon x\in M\}$ such that 
	\begin{itemize}
		\item $\mu_x$ is supported on $f^{-1}[\{x\}]$ for every $x\in M$,
		\item $\mu(B) = \int_{x\in M} \mu_x(B) \ d\mu f^{-1}$ for each measurable $B\subseteq K$,
		\item the mapping $x \mapsto \mu_x(B)$ is measurable for each measurable $B\subseteq K$.
	\end{itemize}
\end{thm}
The family $\{\mu_x\colon x\in M\}$ is called \emph{a disintegration} of $\mu$. For our purposes this theorem will not be of big use, since we are rather interested in compact spaces which are not necessarily separable.

\section{Fibers with non-separable measures.}\label{non-separable}

A \emph{fiber} of a mapping $f\colon K \to M$ is a set of the form $f^{-1}[\{t\}]$, where $t\in M$. In what follows we shall constantly abuse notation by writing $f^{-1}(t)$ instead of $f^{-1}[\{t\}]$. More generally, a \emph{cylinder} of $f$ is a
set of the form $f^{-1}[X]$, where $X\subseteq M$.

We begin with a fact which is in principle \cite[Corollary 2.13]{Tkachenko}. We prove it by a slightly different method. 
\begin{fact}\label{big fibered}
	Assume that $K$ is a compact space and $f\colon K \to M$ is a continuous function into a compact metric space. If $K$ can be mapped continuously onto $[0,1]^{\omega_1}$, then there is $t\in M$ such that $f^{-1}(t)$ can be mapped continuously onto $[0,1]^{\omega_1}$.
\end{fact}

\begin{proof}
	Let $g\colon K \to [0,1]^{\omega_1}$ be a continuous surjection. For each $\alpha<\omega_1$ let \[A^0_\alpha = g^{-1}[\{x \colon x(\alpha)\geq 2/3\}] \mbox{    and    } A^1_\alpha = g^{-1}[\{x\colon x (\alpha)\leq 1/3\}].\]
\noindent \textbf{Claim.} Assume that $F\subseteq M$ is a closed set and $\alpha<\omega_1$ is such that \[ \bigcap_{i\in I} g^{-1}[A^{\epsilon(i)}_i] \cap f^{-1}[F] \ne \emptyset \] for each finite $I\subseteq [\alpha,\omega_1)$ and $\epsilon\colon I\to\{0,1\}$. Then there is a closed $F'\subseteq F$, $\mathrm{diam}(F')\leq 
\mathrm{diam}(F)/2$ and $\beta\geq \alpha$ such that \[ \bigcap_{i\in I} g^{-1}[A^{\epsilon(i)}_i] \cap f^{-1}[F'] \ne \emptyset\] for each finite $I\subseteq [\beta,\omega_1)$ and $\epsilon\colon I\to\{0,1\}$.\medskip

Suppose the contrary. Let $F=F_0\cup F_1$, where $\mathrm{diam}(F_i)\leq \mathrm{diam}(F)/2$, $F_i$ - closed for $i\in\{0,1\}$. Suppose that there is finite $I\subseteq [\alpha,\omega_1)$ and $\epsilon\colon I\to \{0,1\}$ such that 
\[ \bigcap_{i\in I} g^{-1}[A^{\epsilon(i)}_i] \cap f^{-1}[F_0] = \emptyset. \]
(Otherwise, there is nothing to prove.) Let $\beta>\max I$ and suppose, towards contradiction, that there is a finite $J\subseteq [\beta,\omega_1)$ and $\delta\colon J\to \{0,1\}$ such that
\[ \bigcap_{i\in J} g^{-1}[A^{\delta(i)}_i] \cap f^{-1}[F_1] = \emptyset. \]
Then 
\[ \bigcap_{i\in J} g^{-1}[A^{\delta(i)}_i] \cap f^{-1}[F] \subseteq f^{-1}[F_0]. \]
and thus
\[ \bigcap_{i\in I \cup J} g^{-1}[A^{\delta\cup\epsilon(i)}_i] \cap f^{-1}[F] =\emptyset, \]
a contradiction.
\medskip

Notice that \[ \bigcap_{i\in I} g^{-1}[A^{\epsilon(i)}_i] \ne \emptyset \] for every finite $I\subseteq \omega_1$ and every $\epsilon\colon I\to \{0,1\}$. Hence, starting with $F_0=K$ and subsequently using the claim we can find inductively a decreasing sequence $(F_n)$ of closed sets. If $t$ is such that $\{t\} = \bigcap F_n$, then $t$ has the desired property. 
\end{proof}

Let $\mu$ be a measure on a compact space $K$ and let $\mathcal{E}$ be a family of measurable subsets of $K$.
We say that a measure $\mu$ is \emph{approximated by} $\mathcal{E}$ if 
\begin{equation}\label{*}
\inf\{\mu(A \bigtriangleup E)\colon E\in \mathcal{E}\} = 0 
\end{equation}
for each $\mu$-measurable $A\subseteq K$.
The measure $\mu$ is \emph{determined} by $\mathcal{E}$ if
\begin{equation}
 \mu(U) = \sup\{\mu(E)\colon E\in\mathcal{E} \mbox{ and } E\subseteq U\} 
\end{equation}\label{**}
for each open $U\subseteq K$ and finally $\mu$ is \emph{almost-determined} if
\begin{equation}\label{***}
 \mu(U) = \sup\{\mu(E)\colon E\in\mathcal{E} \mbox{ and } \mu(E\setminus U)=0\} 
\end{equation}
for each open $U\subseteq K$. 
In what follows we will consider zerodimensional spaces. Note that in this case the family of clopen subsets approximates $\mu$ and therefore to check if $\mathcal{E}$ approximates $\mu$ it is enough to check if (\ref{*}) holds for all clopen sets $A$.
Similarly, we can check if $\mu$ is (almost) determined by $\mathcal{E}$ by considering (\ref{**}) (and (\ref{***})) only for clopen sets $U$. 

A measure $\mu$ is
of \emph{Maharam type $\kappa$} if $\kappa$ is the minimal size of a family approximating $\mu$. We say that $\mu$ is \emph{separable} if it is of Maharam type at most $\omega$. Notice that a measure $\mu$ is non-separable if there is a family
$\{A_\alpha\colon \alpha<\omega_1\}$ and $\varepsilon>0$ such that $\mu(A_\alpha \bigtriangleup A_\beta)>\varepsilon$ for each $\alpha\ne \beta$.

Note that by saying that $K$ \emph{carries} a measure $\mu$ we do not necessarily assume that $\mu$ is strictly positive on $K$, i.e. that $\mu(U)>0$ for every non-empty open $U\subseteq K$. If $\mu$ is strictly positive on $K$, then we say that $K$ \emph{supports} $\mu$.

We say that
$\omega_1$ is \emph{a precaliber} of a Boolean algebra $\mathfrak{A}$ if for every uncountable family of nonzero elements of $\mathfrak{A}$ there is an uncountable subfamily which is centered, i.e. each of its finite subfamilies has nonempty
intersection. For a cardinal number $\kappa$ let $\lambda_{\kappa}$ be the standard measure on $2^{\kappa}$. 
If $\mathcal{I}$ is an ideal of subsets of $K$, then
\[ \mathrm{cov}(\mathcal{I}) = \min\{|\mathcal{A}|\colon \mathcal{A}\subseteq \mathcal{I}, \ \bigcup \mathcal{A} = K\}. \]
By $\mathcal{N}_\kappa$ we will mean the $\sigma$-ideal of $\lambda_\kappa$-null sets and $\mathcal{N} = \mathcal{N}_\omega$.

\begin{thm}[see e.g. \cite{Kunen-vanMill}] \label{kunen-vanmill}
	The following are equivalent:
	\begin{enumerate}
		\item $\mathsf{MA}_{\omega_1}(\mathrm{ma})$ holds,
		\item $\mathrm{cov}(\mathcal{N}_{\omega_1})>\omega_1$,
		\item $\omega_1$ is a precaliber of measure algebras.
	\end{enumerate}
\end{thm}

The following theorem was mentioned in the introduction.

\begin{thm}[Fremlin, \cite{Fremlin97}] \label{Fremlin}
	$\mathsf{MA}_{\omega_1}$ implies that $K$ can be mapped continuously onto $[0,1]^{\omega_1}$ iff $K$ carries a non-separable measure.
\end{thm}

Theorem \ref{Fremlin} and Theorem \ref{big fibered} implies that under $\mathsf{MA}_{\omega_1}$ if $K$ is a compact space carrying a non-separable measure and $f\colon K \to 2^\omega$ is continuous, then $f^{-1}(t)$ carries a non-separable measure
for some $t\in 2^\omega$. This statement cannot be proved in $\mathsf{ZFC}$: under $\mathrm{cov}(\mathcal{N})=\omega_1$ there is a metrizably-fibered compact space supporting a non-separable measure (see Example \ref{kunen}). However, it can be
generalized in the following way:

\begin{thm}\label{measure fiber} Assume $\omega_1$ is a precaliber of measure algebras. Suppose that $\mu$ is a measure on a compact zerodimensional space $K$, $\mu$ is non-separable and $f\colon K\to 2^\omega$ is a
	continuous mapping. Then there is $t\in 2^\omega$ such that $f^{-1}(t)$ carries a non-separable measure.
\end{thm}

According to Theorem \ref{kunen-vanmill} the assumption used in the above theorem is weaker than $\mathsf{MA}_{\omega_1}$. There are many models in which it is satisfied but $\mathsf{MA}_{\omega_1}$ is not (e.g. the classical random model). Also, it is
consistent with the existence of a compact space which supports a non-separable measure but which cannot be mapped continuously onto $[0,1]^{\omega_1}$ (see \cite[Theorem 6.2]{Plebanek}). Theorem \ref{measure fiber} implies that such spaces (under
$\mathsf{MA}_{\omega_1}(\mathrm{ma})$) cannot have fibers carrying only separable measures. 

On the other hand, since $\mathrm{cov}(\mathcal{N})\geq \mathrm{cov}(\mathcal{N}_{\omega_1})$ (see \cite[Fact 4.1]{Kraszewski}), the assumption used in Theorem \ref{measure fiber} is stronger than $\mathrm{cov}(\mathcal{N})>\omega_1$. We do not know if it can
be weakened to $\mathrm{cov}(\mathcal{N})>\omega_1$ (see also Example \ref{kunen}) but Example \ref{kunen} shows that $\mathrm{cov}(\mathcal{N})>\omega_1$ is the weakest axiom we can try to use to prove Theorem \ref{measure fiber}. 



Let $\mu$ be a measure on a compact zerodimensional space $K$. Let $f\colon K \to 2^\omega$ be a continuous mapping.
For $\sigma\in 2^{<\omega}$ define
\[ \mu_\sigma(A)=\frac{\mu(A\cap f^{-1}\big[ [\sigma] \big])}{\mu(f^{-1}\big[ [\sigma] \big])}, \]
where $[\sigma] = \{t\in 2^\omega\colon \sigma \subseteq t\}$. 
Recall that a sequence of measures $(\mu_n)$ on a zerodimensional compact $K$ converges to $\mu$ in weak$^*$ topology if $\lim_{n\to\infty} \mu_n(A) = \mu(A)$ for each clopen $A\subseteq K$.
Since the family of probability measures is compact in weak$^*$ topology (by the Banach-Alaoglu Theorem \cite[Chapter II]{Diestel84})
for each $t\in 2^\omega$ we can choose a weak$^*$-accumulation point $\mu_t$ of the set $\{\mu_{t|n}\colon n\in \omega\}$.
Notice that for each $t\in 2^\omega$ we have $\mu_t(f^{-1}(t))=1$.


Denote by $\nu$ the measure given by $\nu(A) = \mu f^{-1}[A]$. 


\begin{prop}\label{disintegration}
	Assume that $\mu$ and $K$ are as above. Let $A\subseteq K$ be such that $\mu(A)<\varepsilon$. Then there is a closed set $F_A \subseteq 2^\omega$ such that $\nu(F_A)>0$ and whenever $t\in F_A$ there is $N_A \in \omega$ such that $\mu_{t|n}(A)<\varepsilon$ for each
	$n>N_A$.
\end{prop}

\begin{proof}
	Assume  $\mu(A)=r<\varepsilon$. For $n\in \omega$ let $g_n\colon 2^\omega \to [0,1]$ be defined by
	\[ g_n(t) = \mu_{t|n}(A). \]
	Clearly, $g_n$ is measurable for each $n$ and so $g = \limsup_n g_n$ is measurable, too. Hence,
	\[ H = \{t\in 2^\omega\colon \exists^\infty n  \ \mu_{t|n}(A)\geq \varepsilon\} = g^{-1}[\varepsilon,1] \]
	is measurable. We are going to show that $\nu(H)<1$. Then, any closed $\nu$-positive set $F_A$ such that $F_A \subseteq 2^\omega \setminus H$ would be as desired. Indeed, if $t\notin H$, then we can find $N_A\in \omega$ such that
	$\mu_{t|n}(A)<\varepsilon$ for each $n>N_A$.

	So, suppose towards contradiction that $\nu(H)=1$. Let $\mathcal{C}$ be the family of all clopen sets $C\subseteq 2^\omega$ such that $\mu(A \cap f^{-1}[C])\geq \varepsilon \nu(C)$ whenever $C\in \mathcal{C}$. 
	Notice that $\mathcal{C}$ is a Vitali covering of $H$, i.e. for each $t\in H$ and each
	$\delta>0$, there is $C\in \mathcal{C}$ such that $t\in C$ and $C$ is a ball of radius at most $\delta$. By Vitali Covering Theorem (see e.g. \cite[261B]{Fremlin-MT2}) there is a countable family $\mathcal{C}'\subseteq \mathcal{C}$ of pairwise disjoint sets such that $\nu(H \setminus
\bigcup \mathcal{C}')=0$. 

As $\nu(H)=1$, there is a finite family $\mathcal{F}\subseteq \mathcal{C}'$ such that $\nu(\bigcup \mathcal{F})>r/\varepsilon$. 
Since $\mathcal{C}$ is closed under finite unions of disjoint sets, $C = \bigcup \mathcal{F}\in \mathcal{C}$. 
But then
	\[ \mu(A) \geq \mu(A \cap f^{-1}[C])\geq \varepsilon\cdot \nu(C)>r, \]
	a contradiction.
\end{proof}



	
\begin{lem} \label{x}
	Assume $\omega_1$ is a precaliber of measure algebras. Let $K$ and $\mu$ be as above, $\varepsilon>0$ and let $\mathcal{A}$ be a family of subsets of $K$ such that $\mu(A)<\varepsilon$ for each $A\in\mathcal{A}$. 
Then, there is $t\in 2^\omega$, $N\in \omega$ and an uncountable family $\mathcal{A}'\subseteq \mathcal{A}$ such that $\mu_{t|n}(A)<\varepsilon$ for every $n>N$ and $A\in\mathcal{A}'$.
\end{lem}

\begin{proof}
	For $A\in \mathcal{A}$ let $F_A$ and $N_A$ be as in Lemma \ref{disintegration}. Without loss of generality, refining $\mathcal{A}$ if needed, we can assume that $N_A = N$ for some $N$ and each $A\in \mathcal{A}$.
Since $\omega_1$ is a precaliber of measure algebras, there is an uncountable $\mathcal{A}'\subseteq \mathcal{A}$ such that $\{F_A\colon A\in\mathcal{A}'\}$ is centered. By compactness, there is $t\in \bigcap_{A\in\mathcal{A}'} F_A$. Clearly,
$\mu_{t|n}(A)<\varepsilon$ for each $A \in \mathcal{A}'$ and $n>N$. 
\end{proof}

Recall that a set $D\subseteq 2^\kappa$ \emph{depends} on a set $I\subseteq \kappa$ if $x\in D \iff x_{|I} \in D_{|I}$ for each $x\in 2^\kappa$. A closed set $F\subseteq 2^\kappa$ is a \emph{zero set} if it depends on countably many coordinates.
For a measure $\mu$ defined on a space $K$, denote by $\mathfrak{M}(\mu)$ its measure algebra, i.e. $\mathfrak{M}(\mu) = \mathrm{Bor}(K)_{/ \mu=0}$.
Let $\mathfrak{M} = \mathfrak{M}(\lambda_{\omega})$. A function $\varphi\colon \mathfrak{M}(\mu)\to \mathfrak{M}(\nu)$ is a \emph{measure-isomorphism} if it is a Boolean isomorphism and $\nu(\varphi(A)) =\mu(A)$ for each $A\in \mathfrak{M}(\mu)$.
A measure $\mu$ is \emph{homogeneous} if $\mathfrak{M}(\mu) = \mathfrak{M}(\mu_{| A})$ for every $\mu$-positive $A\subseteq K$. 

It will be convenient to formulate the following corollary of Maharam's theorem:

\begin{thm}[see e.g. \cite{Fremlin-MA}[Theorem 3.9, Theorem 3.10]\label{maharam}
If $\mu$ is homogeneous and of Maharam type $\kappa$, then there is a measure-isomorphism 
$\varphi\colon \mathfrak{M}(\mu) \to \mathfrak{M}(\lambda_{\kappa})$. Moreover, if $K$ carries a measure $\mu$ of Maharam type $\kappa$, then there if a closed $F\subseteq K$  such that $\mu_{| F}$ is a homogeneous measure of Maharam type $\kappa$.
\end{thm}

\begin{lem}\label{independent} Assume $\mu$ is a non-separable measure on $K$ and let $\mathcal{S}$ be a countable family of measurable subsets of $K$. There is a family $\{B_\alpha\colon \alpha<\omega_1\}$ of measurable subsets of $K$ such that \[ \mu\big(S \cap (B_\alpha \bigtriangleup
	B_\beta)\big) = \frac{1}{2}\mu(S) \]
for every $S\in \mathcal{S}$ and $\alpha<\beta<\omega_1$.
\end{lem}

\begin{proof}
	First, note that according to Theorem \ref{maharam} we may assume (considering a closed subspace of $K$ instead of $K$, if needed) that $\mu$ is non-separable and homogeneous and so there is a measure-preserving isomorphism $\varphi\colon \mathfrak{M}(\mu) \to \mathfrak{M}(\lambda_\kappa)$ for $\kappa\geq
	\omega_1$.

Let $D\subseteq 2^{\omega_1}$ be measurable. Notice that there is a countable $I_D\subseteq \kappa$ such that \[ \lambda_{\kappa}(D \cap E) = \lambda_{\kappa}(D) \cdot \lambda_{\kappa}(E) \]
for every measurable $E\subseteq 2^{\kappa}$ depending on $\kappa\setminus I_D$. Indeed, since $\mu$ is inner regular with respect to zero sets (see \cite[Theorem 416U]{Fremlin-MT4}), for every $n$ there is a set $D_n$ such that $D_n\subseteq D$,
$\lambda_{\kappa}(D\setminus
D_n)<1/n$ and $D_n$ depends on a countable set $I_n$. Let $I_D = \bigcup_n I_n$. 

Denote $C_\alpha = \{x\in 2^{\kappa}\colon x(\alpha)=1\}$ and notice that if $\alpha$, $\beta\notin I_D$, then \[ \lambda_{\kappa}\big(D\cap (C_\alpha \bigtriangleup C_\beta)\big)=\frac{1}{2}\mu(D). \]
So, for a $\mu$-measurable $S\subseteq K$ and for $\alpha$, $\beta\notin I_{\varphi(S)}$ we have
\[ \mu\big(S \cap (\varphi^{-1}(C_\alpha) \bigtriangleup \varphi^{-1}(C_\beta))\big) = \frac{1}{2}\mu(S). \]
	Now, let \[ I=\bigcup \{I_{\varphi(S)}\colon S\in \mathcal{S}\}. \]
	Of course $I\subseteq \kappa$ is countable. Choose $\Lambda\subseteq \kappa\setminus I$ of size $\omega_1$
	and finally let \[ \{B_\alpha\colon \alpha<\omega_1\} = \{\varphi^{-1}(C_\beta)\colon \beta\in \Lambda\}. \]
	\end{proof}

\begin{proof} (of Theorem \ref{measure fiber})
	Apply Lemma \ref{independent} to $\mu$ and $\mathcal{S} = \{f^{-1}\big[ [\sigma] \big]\colon \sigma\in 2^{<\omega}\}$ to obtain appropriate $\{B_\alpha\colon \alpha<\omega_1\}$. Notice that for every $\sigma\in 2^{<\omega}$ we have
	\[ \mu_\sigma (B_\alpha \bigtriangleup B_\beta) = \frac{1}{2}. \]

	For each $\alpha$ find a clopen set $A_\alpha\subseteq K$ such that 
	\[ \mu(A_\alpha \bigtriangleup B_\alpha)<\frac{1}{8}. \]

	By Lemma \ref{x} applied to the family $\{A_\alpha \bigtriangleup B_\alpha\colon \alpha\in \omega_1\}$ there is $t\in 2^\omega$, $N\in \omega$ and an uncountable $\Lambda\subseteq \omega_1$ such that 
	\[ \mu_{t|n}(A_\alpha \bigtriangleup B_\alpha)<\frac{1}{8} \]
	for every $n>N$ and $\alpha\in\Lambda$. 
	
	Then, for each $\alpha\ne \beta\in \Lambda$ and $n>N$
	\[ \mu_{t|n}(A_\alpha \bigtriangleup A_\beta)> \frac{1}{4}. \]

	Let $\mu_t$ be an accumulation point (in weak$^*$ topology) of the set $\{\mu_{t|n}\colon n>N\}$. 
	Since $(A_\alpha\bigtriangleup A_\beta)$ is clopen for each $\alpha$, $\beta<\omega_1$,
	\[ \mu_t(A_\alpha\bigtriangleup A_\beta)\geq \frac{1}{4} \]
	for each $\alpha\ne \beta\in \Lambda$. Therefore, $\mu_t$ is non-separable.
\end{proof}

\section{Spaces with non-scattered fibers}\label{scattered}

In this section we will prove in $\mathsf{ZFC}$ that zerodimensional spaces with scattered fibers do not carry non-separable measures. 
Since there is, consistently, a compact space which supports a non-separable measure and whose fibers are homeomorphic to $2^\omega$ (see Example \ref{kunen}) one cannot hope to
prove a stronger $\mathsf{ZFC}$ result of this sort.

\begin{thm}\label{perfect fiber}
	Assume that a compact zerodimensional space $K$ carries a non-separable measure $\mu$ and $f\colon K\to 2^\omega$ is a continuous mapping. Then there is $t\in 2^\omega$ such that $f^{-1}(t)$ is not scattered.
\end{thm}

We will need two lemmas.

\begin{lem} \label{propertyP}
	Assume that $\mu$ is a measure on $K$. Let $f\colon K\to 2^\omega$ be a continuous mapping. If $\mu$ is almost determined by closed cylinders, then $\mu$ is separable.
\end{lem}
\begin{proof}
	We will show that the family of clopen cylinders approximates $\mu$. Of course, each closed subset of $2^\omega$ is a countable intersection of clopens. So, if
	$F\subseteq 2^\omega$ is
	closed, then 
	\[ \mu(f^{-1}[F]) = \inf\{\mu(f^{-1}[C])\colon F\subseteq C, C\mbox{ is clopen}\}. \]
	Let $A\subseteq K$ be $\mu$-measurable and let $\varepsilon>0$. By the assumption there is a closed $F\subseteq 2^\omega$ such that $\mu(f^{-1}[F]\setminus A)=0$ and $\mu(A\setminus f^{-1}[F])<\varepsilon/2$. Take a clopen set $C$ such that
	$F\subseteq C$ and 	$\mu(f^{-1}[C]) - \mu(f^{-1}[F])<\varepsilon/2$. Then 
	\[ \mu(A \bigtriangleup f^{-1}[C])<\varepsilon. \]
	As $A$ and $\varepsilon>0$ were chosen arbitrarily, $\mu$ is approximated by the clopen cylinders and so $\mu$ is separable.
\end{proof}

\begin{lem} \label{propertyPP}
	Assume that $K$ is zerodimensional and $f\colon K \to 2^\omega$ is a continuous mapping. Suppose that $\mu$ is not almost determined by closed cylinders. Then there is a clopen set $C\subseteq K$ and a closed set
	$F\subseteq 2^\omega$ such that
\begin{enumerate}
	\item $\mu(f^{-1}[F])>0$,
	\item whenever $G\subseteq F$ is a closed set and $\mu(f^{-1}[G])>0$, then $\mu(f^{-1}[G]\cap C)>0$ and $\mu(f^{-1}[G]\setminus C)>0$. 
\end{enumerate}
\end{lem}
\begin{proof}
Let $C\subseteq K$ be a $\mu$-positive clopen. Assume that there is no closed $F\subseteq 2^\omega$ satisfying the above properties for $C$. 
	Let $\mathcal{F}$ be a maximal pairwise disjoint family of closed subsets of $2^\omega$ such that for each $F\in \mathcal{F}$ we have $\mu(f^{-1}[F])>0$ but either $\mu(f^{-1}[F]\cap C)=0$ or $\mu(f^{-1}[F]\setminus C)=0$. 

Then $\mu(C \setminus \bigcup_{F\in\mathcal{F}} f^{-1}[F]) = 0$. Otherwise,  we could find a closed $F\subseteq 2^\omega$ disjoint with all the members of $\mathcal{F}$ and such that $\mu(f^{-1}[F]\cap C)>0$. As $F$ does not satisfy (2), we could find a closed $G\subseteq F$ such that $\mu(f^{-1}[G])>0$ but such that either $\mu(f^{-1}[G]\cap C)=0$ or $\mu(f^{-1}[G]\setminus C)=0$. We could add $G$ to $\mathcal{F}$ violating its maximality.

Now, let $\mathcal{F}'$ be the family of all finite unions of elements of $\{F\in \mathcal{F}\colon \mu(f^{-1}[F]\setminus C)=0\}$ and notice that 
\[ \mu(C) = \sup\{\mu(f^{-1}[F])\colon F\in \mathcal{F}'\} \]
and for each $F\in \mathcal{F}'$
\[ \mu(f^{-1}[F]\setminus C) = 0. \]
\medskip

Therefore, if there is no clopen for which we can find a closed set satisfying (1) and (2), then $\mu$
is almost determined by closed cylinders, a contradiction. 
\end{proof}

\begin{proof} (of Theorem \ref{perfect fiber})
	Without loss of generality we can assume that $\mu_{|A}$ is non-separable for each $\mu$-positive $A\subseteq K$. Indeed, using Theorem \ref{maharam} we can find a closed $F\subseteq K$ such that $\mu_{| F}$ is homogeneous. Then, instead of $K$
	we may consider the support of $\mu_{|
	F}$ which clearly has the desired property. 

We will inductively construct a family $\{B_\tau\colon \tau\in 2^{<\omega}\}$ of clopen subsets of $K$ and a sequence $(F_n)_n$ of closed subsets of $2^\omega$ such that
\begin{enumerate}
	\item $(B_\tau \cap f^{-1}[F_n]) \subseteq (B_\sigma \cap f^{-1}[F_n])$ if $\sigma\subseteq \tau$ and $\tau\in 2^n$,
	\item $(B_{\tau\smallfrown 0} \cap B_{\tau \smallfrown 1}) \cap f^{-1}[F_{n+1}] = \emptyset$ for $\tau\in 2^n$,
	\item $\mu(f^{-1}[F_n])>0$ and $\mu(B_\tau \cap f^{-1}[G])>0$ if $\tau\in 2^n$, $G\subseteq F_n$ is closed and $\mu(f^{-1}[G])>0$,
	\item $F_{n+1}\subseteq F_n$ for each $n$ and $\lim_{n\to\infty} \mathrm{diam}(F_n) = 0$.
\end{enumerate}

First, assume that we have a family as above. Let $t$ be such that  $\{t\} = \bigcap F_n$. For each $s\in 2^\omega$ by (3) and by compactness we can choose $b_s \in  \bigcap_n B_{s|n} \cap \bigcap_n f^{-1}[F_n] = \bigcap_n B_{s|n} \cap f^{-1}(t)$. Let $B = \{b_s\colon s\in 2^\omega\} \subseteq f^{-1}(t)$ and notice
that the function $g\colon B \to 2^\omega$ given by $g(b_s)=s$ is continuous (if $\sigma\in 2^{<\omega}$, then $g^{-1}\big[ [\sigma] \big] = B \cap B_\sigma$, because of (1) and (2)). Hence, $B$ can be mapped continuously onto $2^\omega$ and so it is not scattered. 

To complete the proof we perform the promised inductive construction. Let $B_\emptyset=K$ and $F_0 = 2^\omega$. 

Assume that we have have $F_n$ and $B_\tau$ for all $\tau\in 2^{n}$. Enumerate $2^{n} = \{\tau_k\colon k < 2^{n}\}$. Let $H_0 = B_{\tau_0}\cap f^{-1}[F_n]$. According to (3) and to our preliminary assumption $\mu_{|H_0}$ is not separable and so it is not almost determined by closed cylinders. So, we can
apply Lemma \ref{propertyPP} to $H$ and $f_{|H_0}$ to find closed $G_0\subseteq
F_n$ such that $\mu(f^{-1}[G_0]\cap H_0)>0$ and a clopen subset $C_0\subseteq K$ such that
\begin{itemize}
	\item $\mu\left(f^{-1}[G] \cap (C_0 \cap H_0)\right)>0$ and $\mu\left(f^{-1}[G] \cap (H_0 \setminus C_0)\right)>0$ for each closed $G\subseteq G_0$ such that $\mu(f^{-1}[G])>0$,
	\item $\mathrm{diam}(G_0)<1/(n+1)$.
\end{itemize}
Then, proceed in the same manner, letting $H_{k+1} = B_{\tau_{k+1}}\cap f^{-1}[G_k]$ and finding $G_{k+1} \subseteq G_k$, $C_k$ for $k<2^{n}$.
Finally, define 
\begin{itemize}
	\item $F_{n+1} = G_{2^{n}-1}$, 
	\item $B_{\tau_k\smallfrown 0} = C_k \cap B_{\tau_k}$,
	\item $B_{\tau_k\smallfrown 1} = B_{\tau_k} \setminus C_k$.
\end{itemize}
It is straightforward to check that in this way we obtain the desired properties.
\end{proof}



	One may get an impression that being a measure almost determined by a family $\mathcal{E}$ is just a slightly weaker property than being a measure determined by $\mathcal{E}$. This impression is rather misleading. For example, in our
	setting every measure which is determined by closed cylinders is in fact determined by clopen cylinders. It follows from the fact that every closed $F\subseteq
	2^\omega$ can be written as $F = \bigcap C_n$, where $C_n$ is clopen for each $n$. If $U$ is an open set and $f^{-1}[F]\subseteq U$, then, by compactness, there is $n$ such that $f^{-1}[C_n]\subseteq U$. So, every measure determined by closed
	cylinders is in fact determined by a countable family and, consequently, its support is separable. In contrast, spaces supporting measures which are almost determined by closed cylinders may be very far from being separable (see Example \ref{almost determined}). 

	There are (at least consistently) spaces with scattered fibers which support measures and are non-separable (see remarks at the end of Section \ref{examples}). However, such spaces cannot be finitely-fibered.

	\begin{thm}\label{fin fibered}
  Assume $K$ is a zerodimensional compact space supporting a strictly positive measure $\mu$ and let $f\colon K\to	2^\omega$ be a finite-to-one continuous map. 
  Then $K$ is separable.
\end{thm}
\begin{proof}
	Suppose that $K$ is non-separable and let $f\colon K \to 2^\omega$ be a continuous mapping. 
  Without loss of generality we can assume that if $A\subseteq K$ is $\mu$-positive, then $A$ is not separable. Indeed, let $\mathcal{B}$ be a maximal pairwise disjoint family of $\mu$-positive sets $B$ such that $B$ is separable. The
	family $\mathcal{B}$ is clearly countable and so $B' = \bigcup \mathcal{B}$ contains a countable dense set $D$. If $\mu(K\setminus B')=0$, then $U\cap B'\ne \emptyset$ for each non-empty open $U\subseteq K$. Consequently, $U\cap D\ne \emptyset$
	for every non-empty open $U\subseteq K$ and
	so $K$ would be separable. So $\mu(K\setminus B')>0$ and there is a closed $\mu$-positive set $K' \subseteq K\setminus B'$. Then $A$ is non-separable for
	each $\mu$-positive $A\subseteq K'$. So, we can consider $K'$ instead of $K$ if needed.

  Since $K$ is not separable, $\mu$ is not determined by clopen cylinders and, according to above remarks, it is neither by closed cylinders. 
  It means that there is a closed $L\subseteq K$ such that $\mu(L)>0$ and for each $t\in f[L]$ we have $f^{-1}(t)\setminus L\ne \emptyset$. We can additionally assume that $L$ supports the measure $\mu_{|L}$ (throwing out all
  clopens $C$ such that $\mu(L\cap C)=0$). Since $L$ is not separable, the measure $\mu_{|L}$ is not determined by $\{f^{-1}[G]\cap L \colon G = \overline{G} \subseteq 2^\omega\}$.

  Subsequently using these remarks we can construct a sequence $(L_n)_n$ of closed sets such that $L_0 = K$ and for every $n$
  \begin{itemize}
	  \item $L_{n+1}\subseteq L_n$,
	  \item $\mu(L_n)>0$ for each $n$ and $L_n$ supports $\mu_{| L_n}$,
	  \item if $t\in f[L_n]$, then $f^{-1}(t)\cap (L_n \setminus L_{n+1}) \ne \emptyset$.
  \end{itemize}

  Let $x\in \bigcap L_n$ and let $t = f(x)$. The set $f^{-1}(t)$ is infinite since $f^{-1}(t)\cap (L_n \setminus L_{n+1})\ne \emptyset$ for every $n$. So $f$ is not a finite-to-one map.
\end{proof}

One cannot omit the assumption that $K$ supports a measure. E.g. let $K$ be the Alexandroff duplicate of $2^\omega$, i.e. $K = (2^\omega\times \{0\}) \cup (2^\omega \times \{1\})$ and the topology is generated by the sets of the form $\{(x,1)\}$ for $x\in 2^\omega$ and $C\times \{0,1\} \setminus \{(x,1)\}$ for a
clopen $C\subseteq 2^\omega$ and $x\in 2^\omega$. It is straightforward to check that $K$ is compact, Hausdorff and non-separable but the fibers of the natural retraction onto $2^\omega\times \{0\}$ are of size 2.

\section{Examples} \label{examples}

Every space can be mapped continuously into a metrizable compact space and all compact spaces which are not scattered can be mapped continuously onto $2^{\omega}$. Most of these mappings are not particularly interesting in context of its fibers.
However, quite often such mapping can be chosen in a more or less \emph{canonical} way. 
If a compact space is constructed as an inverse limit, then usually its first coordinate is compact and metrizable and the projection of the whole space onto it is of course continuous. Also, every Boolean algebra $\mathfrak{A}$ which is not superatomic (i.e. its
Stone space is not scattered) have a Cantor algebra $\mathfrak{C}$ as a subalgebra and then its
Stone space is mapped continuously onto $2^\omega$ by the function
 \[ f(x) = x_{|\mathfrak{C}}. \]

 The properties of fibers of such mappings are sometimes easier to grasp than the \emph{global} properties of the space and studying them can lead to interesting pieces of information about the space itself. In particular, it is not always easy to
 show directly that a space cannot be mapped continuously onto $[0,1]^{\omega_1}$. Usually one has to prove
 that a space possesses some other property which makes such mapping impossible (e.g. is countably tight, Corson compact, hereditary Lindel\"{o}f etc.). Similarly, sometimes it not clear if a given space carries a non-separable measure. Results from
 previous sections provide here yet another tool.

 In this section we will overview several examples which we found interesting. We begin with remarks on the Stone space of $\mathfrak{M}$. Then we present two non-separable metrizably-fibered spaces supporting a measure inspired by a certain
 Kunen's construction. We finish with examples of spaces with scattered fibers of Bell and Todor\v{c}evi\'{c}. 

We will use the fact that every (finitely-additive) measure on a Boolean algebra can be extended uniquely to a Radon measure on its Stone space. If $A\in \mathfrak{A}$, then by $\widehat{A}$ we denote the appropriate clopen subset of
$\mathrm{Stone}(\mathfrak{A})$, but only in case of a possible confusion. Otherwise, we do not distinguish in notation between elements of a Boolean algebra and clopens of its Stone space. For the sake of brevity we are not going to define all the
notions which will appear in what follows, but the reader can easily find the definitions following the given references. 

\begin{example}\label{almost determined} Stone space of the measure algebra.

	Let $K$ be the Stone space of $\mathfrak{M}$. Let $\mathfrak{C}\subseteq \mathfrak{M}$ be the Boolean algebra generated by $\{[C]_{\lambda=0}\colon C\in \mathrm{Clop}(2^\omega)\}$ and let $f\colon K \to 2^\omega$ be defined by $f(x) =
	x_{|\mathfrak{C}}$. Notice that the fibers of $f$ are homeomorphic to each other. Indeed, if $t\in 2^\omega$, then $h\colon f^{-1}(0) \to f^{-1}(t)$ given by $h(x) = x+t$ is a homeomorphism. (Here 0 stands for the element of $2^\omega$ constantly
	equal to 0 and $x+t$ is the ultrafilter consisting of $\{A+t\colon A\in x\}$, where $[F]_{\lambda = 0} + t = [F+t]_{\lambda = 0}$.)

	It is well-known that the space $K$ is not separable and it maps continuously onto $[0,1]^{\mathfrak{c}}$. So, its fibers map continuously onto $[0,1]^{\omega_1}$. However, $K$ supports a measure which is almost determined by closed cylinders. 

	Let $\widehat{\lambda}$ be the unique extension of $\lambda$ to $K$. 
	Let $M\in \mathfrak{M}$ and $\varepsilon>0$. There is $M_0 \in \mathfrak{M}$ such that $M_0 = [F]_{\lambda=0}$ for some closed $F\subseteq 2^\omega$, $M_0
	\subseteq M$ and $\lambda(M\setminus M_0)<\varepsilon$. Then $\widehat{\lambda}(f^{-1}[F]\setminus M_0)=0$ since $\lambda(F)=\widehat{\lambda}(M_0)$. Therefore $\widehat{\lambda}(f^{-1}[F]\setminus M)=0$. But $\widehat{\lambda}(M\setminus
	f^{-1}[F]) \leq \widehat{\lambda}(M\setminus M_0)<\varepsilon$.
\end{example}

We will now present an example of a non-separable space supporting a measure which cannot be mapped continuously onto $[0,1]^{\omega_1}$. The latter property follows from the fact that the presented space is metrizably-fibered (because of
Theorem \ref{big fibered} and the fact that metrizable spaces cannot be mapped continuously onto $[0,1]^{\omega_1}$). We show that one can demand that the supported measure is non-separable. The construction is in the spirit of \cite{Kunen} (although Kunen constructed his space as an inverse limit). 

If $\mathcal{I}$ is an ideal of subsets of $K$, then
\[ \mathrm{add}(\mathcal{I}) = \min\{|\mathcal{A}|\colon \mathcal{A}\subseteq \mathcal{I}, \ \bigcup \mathcal{A}\notin \mathcal{I}\}, \]
\[ \mathrm{cof}(\mathcal{I}) = \min\{|\mathcal{A}|\colon \forall I \in \mathcal{I} \ \exists A\in \mathcal{A} \ I\subseteq A\}. \]

\begin{example}\label{kunen} Non-separable measure with small non-separable support under the assumption $\mathrm{cov}(\mathcal{N})=\omega_1$.

	Let $\{N_\alpha\colon \alpha<\omega_1\} \subseteq \mathcal{N}$ be an increasing family witnessing $\mathrm{cov}(\mathcal{N})=\omega_1$. 
	Let $\mathcal{F} = \{F_\alpha\colon \alpha<\omega_1\}$ be a family of closed subsets of $2^\omega$ such that for each $\alpha<\omega_1$
	\begin{itemize}
		\item[a)] $F_\alpha \cap N_\alpha = \emptyset$,
		\item[b)] $\lambda(F_\alpha)>0$.
	\end{itemize}
	It can be easily seen that such family can be constructed by a transfinite induction.

	Let $\mathfrak{C}$ be the Boolean algebra of clopens of $2^\omega$. Generate a Boolean algebra $\mathfrak{A}$ by $\mathfrak{C}$ and the family $\{F_\alpha\colon \alpha<\omega_1\}$. Denote by $K'$ its Stone space and by $f'\colon K \to
	\mathrm{Stone}(\mathfrak{C}) = 2^\omega$ the (continuous) mapping $f(x)=x_{| \mathfrak{C}}$. The Boolean algebra $\mathfrak{A}$ carries the
	measure $\lambda$ (which does not need to be strictly positive on $\mathfrak{A}$). Let $\widehat{\lambda}$ be the unique extension
	of the measure $\lambda$ to $K'$. Let $K$ be the support of $\widehat{\lambda}$ and let $f = f'_{| K}$. If $t\in 2^\omega$, then $t\in
	N_\alpha$ for some $\alpha<\omega_1$. Hence, $t\notin F_\beta$ for $\beta\geq \alpha$. So, the topology of $f^{-1}(t)$ is generated by countably many sets and, thus, $f^{-1}(t)$ is metrizable.
	
	Moreover, $K$ is not separable. If $X$ is a countable subset of $K$, then $f[X]$ is a countable subset of $2^\omega$ and so we can find
	$\alpha<\omega_1$ such that $f[X]\subseteq N_\alpha$. Then, $F_{\alpha+1}\cap f[X]=\emptyset$ and so $\widehat{F}_{\alpha+1} \cap X = \emptyset$. But $\widehat{\lambda}(\widehat{F}_{\alpha+1})>0$ and therefore $K\cap \widehat{F}_{\alpha+1}$ is
	a non-empty clopen subset of $K$ disjoint from $X$.

	Using this method we can construct a space like above which additionally supports a non-separable measure. It will be more convenient to work in $2^{\omega_1}$ instead of $2^\omega$. Let $\pi\colon 2^{\omega_1} \to 2^\omega$ be the standard
	projection. Notice that the family $\{\pi^{-1}[N_\alpha]\colon \alpha<\omega_1\}$ is a family of $\lambda_{\omega_1}$-null sets covering $2^{\omega_1}$. Consider a family $\mathcal{F} = \{F_\alpha\colon \alpha<\omega_1\}$ of closed
	subsets of $2^{\omega_1}$ such that for each $\alpha<\omega_1$
	\begin{itemize}
		\item[a)] $F_\alpha \cap \pi^{-1}[N_\alpha]=\emptyset$,
		\item[b)] $\inf\{\lambda_{\omega_1}(F_\xi \bigtriangleup F_\alpha)\colon \xi<\alpha\}>0$. 
	\end{itemize}
	As before, it can be obtained by a simple transfinite induction. Condition (b) is easy to achieve since $\lambda_{\omega_1}$ is non-separable and so the families $\{F_\xi\colon \xi<\alpha\}$ do not approximate $\lambda_{\omega_1}$ for
	$\alpha<\omega_1$.

	Generate the Boolean algebra $\mathfrak{A}$ by $\mathrm{Clop}(2^{\omega_1})$ and the family $\{F_\alpha\colon \alpha<\omega_1\}$ and let $K$ be defined in an analogous way as above, so that $\widehat{\lambda}_{\omega_1}$ is
	strictly positive on $K$. If $g\colon K \to 2^{\omega_1}$ is defined by $g(x)=x_{|\mathrm{Clop}(2^{\omega_1})}$ then the mapping $f = \pi\circ g\colon K\to 2^\omega$ is continuous. As before, if $t\in N_\alpha$, then $\widehat{F}_\beta \cap
	f^{-1}(t)=\emptyset$ for every $\beta\geq \alpha$ and so $K$ is metrizably-fibered.

	Clearly, there is $\varepsilon>0$ and $\Lambda\subseteq \omega_1$ of size $\omega_1$ such that 
	\[ \inf\{\lambda_{\omega_1}(F_\xi \bigtriangleup F_\alpha)\colon \xi<\alpha, \xi, \alpha\in \Lambda\} > \varepsilon. \]
	Hence, $\widehat{\lambda}_{\omega_1}(\widehat{F}_\alpha \bigtriangleup \widehat{F}_\beta)>\varepsilon$ for each $\alpha, \beta\in \Lambda$ and so $\widehat{\lambda}$ is non-separable.

	By manipulating the conditions imposed on the family $\mathcal{F}$ one can obtain spaces satisfying various properties: e.g. Kunen in \cite{Kunen} provided  a $\mathsf{CH}$ example of a Corson compact L-space
	supporting a non-separable measure $\mu$ for which $\mu(A)=0$ if and only if
	$A$ is metrizable. In \cite{Kunen-vanMill}[Theorem 1.2] the authors proved that such space can be constructed assuming only that $\mathrm{cof}(\mathcal{N}_{\omega_1})=\omega_1$ (both of these spaces are metrizably-fibered). In
	\cite{Kunen-vanMill}[Theorem 1.1] it is also shown that
	$\mathrm{cov}(\mathcal{N})>\omega_1$ implies that there is a Corson compact space supporting a non-separable measure. It is not clear for us if this space is metrizably-fibered. If not, then perhaps the assumption of Theorem
	\ref{measure fiber} can be relaxed to $\mathrm{cov}(\mathcal{N})>\omega_1$. Other interesting examples inspired by the Kunen's example can be found e.g. in \cite{Plebanek}.  

\end{example}

Now we turn our attention to examples of spaces with scattered fibers. We will start with a technique of \emph{total ideal spaces} developed by Murray Bell which was used to produce several interesting spaces (see \cite{Bell} and
\cite{Bell-gap}). We will overview
two of them. The presentation differs from that of Bell: we find more convenient to see them as Stone spaces of certain Boolean algebras.

\begin{example} Total ideal spaces. \label{total}


In \cite{Bell} the author constructs a compact separable space which cannot be mapped continuously onto $[0,1]^{\omega_1}$ and which does not have a countable $\pi$-base. 

Let $\mathcal{T} = \{T_\alpha\colon \alpha<\omega_1\}$ be any $\subseteq^*$-increasing chain in $P(\omega)$ such that $T_0=\emptyset$. For each $A\subseteq \omega$ let $X_A = \{x\in 2^\omega\colon x(n)=1\mbox{ if }n\in A\}$. Generate a Boolean
algebra $\mathfrak{B}$ of subsets of $2^\omega$ by the family:
\[ \{X_A\colon A =^* T_\alpha\mbox{ for some }\alpha<\omega_1\} \]
and let $K_\mathcal{T}$ be its Stone space.

The space $K_\mathcal{T}$ is separable (and so it supports a measure) and it does not have a countable $\pi$-base (see \cite{Bell} or \cite{MirnaPbn} for the proofs of these statements). 
Bell proved that it does not map continuously onto $[0,1]^{\omega_1}$ by showing that it is scattered in the $G_\delta$ topology. We will show that it is in fact scatteredly fibered. First, notice that the algebra $\mathfrak{C}$ generated by $\{X_A\colon A=^* \emptyset\} \subseteq \mathfrak{B}$ is countable and non-atomic and so it is isomorphic to the Cantor algebra. The mapping $f\colon K_\mathcal{T} \to \mathfrak{C}$, given by
$f(x)=x_{|\mathfrak{C}}$, is
continuous. We will show that fibers of $f$ are homeomorphic to ordinal numbers (and so, in particular, they are scattered). Indeed, let $t\in 2^\omega$ and let
$\alpha<\beta$. If $n>\max (T_\alpha\setminus T_\beta)$, then \[ (f^{-1}([t|n]) \cap T_\alpha) \subseteq (f^{-1}([t|n]) \cap T_\beta)\] unless $f^{-1}\big(([t|n]) \cap T_\beta\big)$ is empty. Hence, if $T_\alpha \cap f^{-1}(t) \ne \emptyset$ and $T_\beta
\cap f^{-1}(t) \ne \emptyset$, then \[ T_\alpha \cap f^{-1}(t) \subseteq T_\beta \cap f^{-1}(t), \]
and thus the topology of $f^{-1}(t)$ is generated by a well-ordered family.

It is not difficult to see that if we apply the above machinery to a family different than $\mathcal{T}$ (but containing all finite sets), then the resulting space will still be separable (see \cite[Proposition 6.5]{MirnaPbn}). In \cite{Bell-gap} Bell modified this technique to provide an interesting example of a non-separable space. 

For $A$, $B\subseteq \omega$ let \[X_{A,B} = \{x\in 2^\omega\colon x(n)=1 \mbox{ if }n\in A\mbox{ and }x(n)=0 \mbox{ if }n\in B\}. \]
Let $\mathcal{G} = \{(L_\alpha, R_\alpha)\colon \alpha<\kappa\}$ be a \emph{pregap} on $\omega$ \emph{of height} $\kappa$, i.e. $(L_\alpha)_{\alpha<\kappa}$, $(R_\alpha)_{\alpha<\kappa}$ are $\subseteq^*$-increasing and $L_\alpha \cap R_\alpha =
\emptyset$ for each $\alpha<\kappa$. We call
$\mathcal{G}$ a \emph{gap} if there is no $L\subseteq \omega$
such that $L_\alpha\subseteq^* L$ and $R_\alpha \cap L =^* \emptyset$ for each $\alpha<\kappa$. A gap $\mathcal{G}$ is \emph{destructible} if there is a ccc forcing $\mathbb{P}$ such that \[\Vdash_{\mathbb{P}} \check{\mathcal{G}} \mbox{ is not a gap.}\]
Note that gaps of height $\kappa>\omega_1$ are always destructible.

Generate a Boolean algebra $\mathfrak{A}_\mathcal{G}$ by the family:
\[  \{X_{A,B}\colon A=L_\alpha, B=R_\alpha \mbox{ for some }\alpha<\kappa \} \]
and let $K_\mathcal{G}$ be its Stone space.

The space $K_\mathcal{G}$ has scattered fibers. It can be proved as in the case of $K_\mathcal{T}$. Hence, $K_\mathcal{G}$ cannot be mapped continuously onto $[0,1]^{\omega_1}$. Bell proved that $K_\mathcal{G}$ is
not separable if and only if $\mathcal{G}$ is a gap. Moreover, if $\mathcal{G}$ is destructible, then $K_\mathcal{G}$ is ccc. Thus, every example of a destructible gap can be translated, using the above machinery, to an example of ccc non-separable space which
cannot be mapped continuously onto $[0,1]^{\omega_1}$. Destructible gaps do not always exist. Actually, Bell remarked that under $\mathsf{OCA}$ there is a very general reason why using this technique one cannot obtain a ccc non-separable space (see
\cite[Fact 2.3]{Bell}). However, since due to Kunen's result (see \cite{Baumgartner84}) $\mathsf{MA}_{\omega_1}$ is consistent
with the existence of gap of height $\omega_2$ (which is, therefore, destructible), it follows that $\mathsf{MA}_{\omega_1}$ is consistent with the existence of a ccc non-separable space which cannot be mapped onto $[0,1]^{\omega_1}$. 
\end{example}

In \cite{Todorcevic}[Theorem 8.4] Todor\v{c}evi\'{c} carried out a $\mathsf{ZFC}$ construction of a ccc non-separable space which cannot be mapped continuously onto $[0,1]^{\omega_1}$. He showed that his space can be mapped continuously onto $2^\omega$ by a
function whose fibers are homeomorphic to ordinal numbers.

\begin{example} Todor\v{c}evi\'{c}'s  ccc non-separable small space.\label{Stevo}

Let $\mathcal{S}$ be the set of \emph{slaloms}, i.e.
\[ \mathcal{S} = \{S\subseteq \omega\times\omega \colon |S(n)|\leq n \}. \]
Let $\Omega = \{(S,n)\colon n\in\omega, \ S\in \mathcal{S}, \ S\subseteq (n\times n)\}$.
For each $A\in \mathcal{S}$ define 
\[ T_A = \{(T,n)\in \Omega \colon A\cap (n\times n) \subseteq T\}. \]
For $(S,n)\in \Omega$ let
\[ T_{(S,n)} = \{(T,m)\in \Omega \colon m\geq n, T\cap (n\times n) = S\}. \]
Let 
\[ \mathcal{Z} = \{S\subseteq \omega\times \omega\colon S\in \mathcal{S}\mbox{ and } \lim_n |S(n)|/n = 0 \}. \]

There is a family $(A_\alpha)_{\alpha<\mathrm{add}(\mathcal{N})}$ of elements of $\mathcal{Z}$ such that $A_\alpha \subseteq^* A_\beta$ if $\alpha<\beta$ and there is no $S\in\mathcal{S}$ such that $A_\alpha \subseteq^* S$ for every $\alpha<\mathrm{add}(\mathcal{N})$ (see \cite{Kunen-Fremlin}). 

Let $\mathcal{A} = \{A\in \mathcal{Z}\colon A =^* A_\alpha \mbox{ for some }\alpha<\mathrm{add}(\mathcal{N})\}$ and 
\[ \mathfrak{T}_\mathcal{A} = {\rm alg}\left(\{T_A \colon A\in \mathcal{A}\} \cup \{T_{(S,n)}\colon (S,n)\in\Omega\}\right). \]
Finally let $K_\mathcal{A}$ be the Stone space of $\mathfrak{T}_\mathcal{A}/{\rm Fin}$. 

Notice that the algebra $\mathfrak{C} = \mathrm{alg}(\{T_{(S,n)}\colon (S,n)\in \Omega\})$ is a countable non-atomic Boolean algebra and hence it is isomorphic to the Cantor algebra.
So, the mapping $f\colon K_\mathcal{A}\to \mathrm{Stone}(\mathfrak{C})$ given by $f(x) = x_{|\mathfrak{C}}$ is a continuous mapping into $2^\omega$. We will show that $K_\mathcal{A}$ has scattered fibers (see also
\cite[Claim 4, Theorem 8.4]{Todorcevic}). 

Let $t\in \mathrm{Stone}(\mathfrak{C})$ and let $A$, $B$ be such that $A=^* A_\alpha$, $B=^* A_\beta$ for some $\alpha\leq \beta < \mathrm{add}(\mathcal{N})$. Suppose that $T_A\cap f^{-1}(t)\ne \emptyset$ and $T_B\cap f^{-1}(t)\ne \emptyset$. There
is $n$ such that $A \setminus (n\times n) \subseteq B$. Let $S\subseteq (n\times n)$ be such that $T_{(S,n)} \in t$. Then $T_A \cap T_{(S',n')} \subseteq T_B \cap T_{(S',n')}$ for each $n'>n$ and $T_{(S',n')}\in t$ and, consequently, $T_A \cap f^{-1}(t) \subseteq
T_B \cap f^{-1}(t)$. Therefore, the topology of $f^{-1}(t)$ is induced by a well-ordered chain and so it is homeomorphic to an ordinal number. 

Todor\v{c}evi\'{c} showed that $K_\mathcal{A}$ is ccc and non-separable, obtaining therefore a $\mathsf{ZFC}$ example of ccc non-separable space without a continuous mapping onto $[0,1]^{\omega_1}$.
\end{example}


One of the motivations to study the fibers of spaces supporting measures was the question about the existence of non-separable spaces supporting measures which cannot be mapped continuously onto $[0,1]^{\omega_1}$. Actually, some of the theorems
presented in this article were proved in an attempt to show that spaces with scattered fibers supporting measures have to be separable (and, in particular, to prove that the space from Example \ref{Stevo} cannot support a measure). Our efforts were
hopeless. First, in \cite{Pbn-Grzes-16}, under $\mathsf{MA}$, the authors provided an example of a non-separable space with scattered fibers supporting a measure (in a sense distilling some ideas present in Example \ref{kunen} and Example \ref{total}). Then, it
turned out that Todor\v{c}evi\'{c}'s construction described in Example \ref{Stevo} can be modified to obtain a non-separable space with scattered fibers supporting a measure
assuming only that $\mathrm{add}(\mathcal{N}) = \mathrm{non}(\mathcal{M})$ (see \cite{Tanmay}). 

We still do not know if there is a $\mathsf{ZFC}$ example of such space:

\begin{prob}
	Is it consistent that spaces with scattered fibers and supporting measures are separable?
\end{prob}

\section{Acknowledgements}
This research was partially done whilst the author was visiting fellow at the Isaac Newton Institute for Mathematical Sciences, Cambridge, in the
programme `Mathematical, Foundational and Computational Aspects of the Higher Infinite' (HIF).
The author would like to thank Grzegorz Plebanek and Piotr Drygier for very helpful discussions on the subject of this paper.

\bibliographystyle{alpha}
\bibliography{smallb}

\end{document}